\documentclass{lms}
\usepackage[utf8]{inputenc}
\usepackage{amssymb,amsmath}
\usepackage{paralist} 
\usepackage{url}
\newtheorem{theorem}{Theorem}[section] 
\newtheorem{lemma}[theorem]{Lemma}     
\newtheorem{corollary}[theorem]{Corollary}

\newnumbered{assertion}{Assertion}    
\newnumbered{conjecture}{Conjecture}  
\newnumbered{definition}{Definition}
\newnumbered{hypothesis}{Hypothesis}
\newnumbered{remark}{Remark}
\newnumbered{note}{Note}
\newnumbered{observation}{Observation}
\newnumbered{problem}{Problem}
\newnumbered{question}{Question}
\newnumbered{example}{Example}

\newcommand{\abs}[1]{\left|#1\right|}
\newcommand{\set}[1]{\left\{#1\right\}} 
\newcommand{\IR}{\mathbb{R}} 
\newcommand{\conv}{\textrm{conv}} 

\title[Tverberg plus constraints]
 {Tverberg plus constraints} 

\author[P. V. M. Blagojevi\'c, F. Frick \& G. M. Ziegler]{Pavle V. M. Blagojevi\'c, Florian Frick and G\"unter M. Ziegler}
 
\classno{52A35 (primary), 55S35 (secondary)} 

\extraline{The research leading to
     these results has received funding from the European Research Council under
     the European Union's Seventh Framework Programme (FP7/2007-2013)/ERC Grant agreement no.~247029-SDModels. The first author was also supported by the grant ON 174008 of the Serbian
     Ministry of Education and Science. The second author is supported by the German Science Foundation DFG via the Berlin Mathematical School.
     The third author is also supported by the DFG Collaborative Research Center TRR~109 
     ``Discretization in Geometry and Dynamics.''}

\begin{document} 
\maketitle

\begin{abstract}
    Many of the strengthenings and extensions of the topological Tverberg theorem 
    can be derived with surprising ease directly from the original theorem:
    For this we introduce a proof technique that combines a concept of ``Tverberg unavoidable subcomplexes'' 
    with the observation that Tverberg points that equalize  
    the distance from such a subcomplex can be obtained from maps to an extended
    target space.
    
    Thus we obtain simple proofs for many variants of the topological Tverberg theorem, 
    such as the colored Tverberg theorem of \v{Z}ivaljevi\'c \& Vre\'cica (1992).
    We also get a new strengthened version
    of the generalized van Kampen--Flores theorem by Sarkaria (1991) and Volovikov (1996),  
    an affine version of their ``$j$-wise disjoint'' Tverberg theorem,
    and a topological version of Sober\'on's (2013) result on Tverberg points with equal barycentric coordinates.
\end{abstract}

\section{Introduction}
Tverberg's 1966 theorem \cite{tverberg:generalisation_radon}, which states that any set of
$(r-1)(d+1)+1$ points in $\IR^d$ can be partitioned into $r$ subsets whose convex hulls intersect,
 is a seminal result that has inspired many interesting
variations, extensions, and analogues---starting with the ``topological Tverberg theorem'' 
of Bárány, Shlosman \& Sz\H{u}cs \cite{barany_schlosman_szucs:toplogical_tverberg} and \"Ozaydin~\cite{oezaydin:equivariant}. 
See Section \ref{sec:tverberg's theorem and its relatives} for a brief review
and Matou\v{s}ek \cite{matousek:borsuk-ulam2} for a friendly textbook treatment.

Here we propose a new, simple and elementary, proof technique that combines a concept of ``Tverberg unavoidable subcomplexes'' 
(which contain at least \emph{one} simplex from each Tverberg partition), see Section~\ref{sec:T-unavoidable-subcomplexes},
with the observation that Tverberg points that equalize  
the distance from such a subcomplex can be obtained from maps to an extended
target space, see Section~\ref{sec:constraining_function}. This technique allows us to derive many of the variations and analogues 
directly from the original topological Tverberg theorem with surprising ease. 
 
For example, our ansatz produces \emph{directly from the topological Tverberg theorem}
a colored version that is stronger than \v{Z}ivaljevi\'c \& Vre\'cica's 1992 ``colored Tverberg's theorem'' \cite{vzivaljevic1992colored}
(but weaker than the \emph{optimal} colored Tverberg theorem \cite{blagojevic2009optimal} from 2009);  
see Section~\ref{sec:colored}.
Similarly we obtain directly from the topological Tverberg theorem 
a strengthened version of the ``generalized van Kampen--Flores theorem'' of Sarkaria \cite{sarkaria:generalized_van-kampen_flores} 
and Volovikov \cite{volovikov:van-kampen_flores} from 1991/1996; see Section~\ref{sec:prescribing_dimensions}.  
As another example, our machinery reproves Sober\'on's 2013 
``Tverberg theorem with equal barycentric coordinates''
and at the same time produces a new topological version of this result; see Section~\ref{sec:Soberon}. 

Our machinery uses the topological Tverberg theorem as a \emph{black box}:
It provides the core for our proofs, but the proof technique relies only on the result, not on its proof.
Thus, for example, an extension of the topological Tverberg theorem to some $r$ that
are not prime powers would immediately yield the same extension for our derived results.
In place of the black box result we can also use the optimal colored Tverberg theorem; in Section~\ref{sec:optimal} we 
thus obtain further results of the ``colored Tverberg'' type.
The black box result can also be replaced by Tverberg's original theorem if the constraint functions are affine; in 
this case no prime power restriction on the number of parts is needed. 
Thus we obtain a new affine version of the topological Tverberg theorem for $j$-wise disjoint simplices; see Section~\ref{sec:j-wise}.
Similarly, from the Tverberg theorem for maps to manifolds by Volovikov \cite{volovikov1996topological}
one ``automatically'' gets extensions of our results for maps to manifolds.%

\section{A brief history of Tverberg type results}
\label{sec:tverberg's theorem and its relatives}

For every drawing of a tetrahedron in the plane either a vertex will end up on top of the opposite face or two opposite edges intersect in the drawing.  
This is a first instance of a Tverberg type result, which holds true in any dimension as was proved by
Johann Radon (1921):
{\em
\begin{compactitem}[\qquad]
\item {\sc Radon's theorem.} For any subset $X\subseteq \IR^d$ with (at least) $d+2$ elements there are disjoint subsets $S$ and $T$ of $X$ with the property that
$\conv(S)\cap\conv(T)\neq\emptyset$.
\end{compactitem}
}
\noindent
Radon's theorem has an equivalent reformulation in terms of affine maps of the
$(d+1)$-\allowbreak dimensional standard simplex $\Delta_{d+1}$ to $\IR^d$:
{\em
\begin{compactitem}[\qquad]
\item {\sc Affine Radon theorem.} 
For every affine map $f\colon\Delta_{d+1}\to\IR^d$ there are disjoint faces $\sigma$ and $\tau$ of $\Delta_{d+1}$ with the property that $f(\sigma)\cap f(\tau)\neq\emptyset$.
\end{compactitem}
}
\noindent
Now we can see many possible ways to extend this basic result.
First, we can ask whether the assumption on $f$ to be an affine map is essential.
Is it enough to assume that $f$ is only continuous?
This question was answered by  Bajm\'oczy \& B\'ar\'any \cite{bajmoczy1979} in 1979 via a clever use of Borsuk--Ulam theorem:
{\em
\begin{compactitem}[\qquad]
\item {\sc Topological Radon theorem.} For any continuous map  $f\colon\Delta_{d+1}\to\IR^d$ there are disjoint faces $\sigma$ and $\tau$ of the simplex $\Delta_{d+1}$ with the property that $f(\sigma)\cap f(\tau)\neq\emptyset$.
\end{compactitem}
}
\noindent
It is also natural to ask for more than two pairwise disjoint subsets of a sufficiently
large set of points $X\subset\IR^d$.
Such a result was first achieved by Birch \cite{Birch59} in 1959 for $d=2$, who also conjectured the
tight result for $d>2$, which was eventually proved by  
Helge Tverberg \cite{tverberg:generalisation_radon} in 1966.
In the equivalent affine version, his result reads as follows: 
{\em
\begin{compactitem}[\qquad]
\item {\sc Affine Tverberg theorem.} Let $d\geq 1$ and $r\geq 2$ be integers, and $N=(r-1)(d+1)$.
For any affine map $f\colon\Delta_N\to\IR^d$ there are $r$ pairwise disjoint faces 
$\sigma_1,\ldots,\sigma_r$ of $\Delta_N$ such that $f(\sigma_1)\cap\cdots\cap f(\sigma_r)\neq\emptyset$.
\end{compactitem}
}
\noindent
The set $\set{\sigma_1, \dots, \sigma_r}$ of disjoint faces 
of $\Delta_N$ whose images intersect is called a 
\emph{Tverberg partition} for $f:\Delta_N\rightarrow\IR^d$.
The dimension of the simplex $\Delta_N$ in the theorem is minimal.

Will the statement of Tverberg's theorem still be true if the map $f$ is only assumed to be continuous?
Surprisingly, the available answers to this question depend on divisibility properties of
the parameter~$r$.
B\'ar\'any, Shlosman \& Sz\H ucs \cite{barany_schlosman_szucs:toplogical_tverberg} in 1981
formulated the corresponding ``topological Tverberg theorem'' 
for continuous maps and proved it for the case $r$ is a prime number.
This was extended to all prime powers $r$ by \"Ozaydin \cite{oezaydin:equivariant} in 1987:
{\em
\begin{compactitem}[\qquad]
\item  {\sc Topological Tverberg theorem.} Let $r\ge2$, $d\ge1$, and $N=(r-1)(d+1)$. 
If $r$ is a prime power, then for every continuous map $f : \Delta_N \to \IR^d$ there are $r$ pairwise disjoint faces $\sigma_1, \dots, \sigma_r$ of $\Delta_N$ such that $f(\sigma_1) \cap \dots \cap f(\sigma_r) \neq \emptyset$.
\end{compactitem}
}
\noindent
The topological Tverberg theorem is derived from the nonexistence of an 
$\mathfrak{S}_r$-equivariant map from the join
$[r]^{*(N+1)}$ into the sphere $S(W_r^{\oplus (d+1)})$ when $r$ is a prime power.
Here $[r]=\{1,\ldots,r\}$ denotes a $0$-dimension simplicial complex of $r$ distinct points with the obvious action of the symmetric group $\mathfrak{S}_r$,  and $W_r:=\{(x_1,\ldots,x_r):x_1+\cdots+x_r=0\}$ is an
 $\mathfrak{S}_r$-representation with an action given by permutation of coordinates.
The fact that an \linebreak$\mathfrak{S}_r$-equivariant map $[r]^{*(N+1)}\to S(W_r^{\oplus (d+1)})$ exists if
$r$ is not a prime power, established in~\cite{Z124},  makes the extension of topological Tverberg theorem to non prime powers into ``one of the most challenging problems in this field'' \cite[Notes to Sect.~6.4]{matousek:borsuk-ulam2}.  

What is next? 
Can we say something about dimensions of simplexes in a Tverberg partition?
What is the minimal dimension of a simplex $\Delta$ 
such that for every mapping to the plane there are two disjoint edges whose images intersect?  
(The answer to this is given by Kuratowski's theorem from graph theory: $K_5$ is not planar, 
so the $4$-simplex will do.)
How about higher-dimensional versions of this result?
The classical van Kampen--Flores theorem \cite{flores:selbstverschlungen} \cite{van1933komplexe} 
from the 1930s provides a first answer to such questions.
{\em
\begin{compactitem}[\qquad]
\item  {\sc van Kampen--Flores theorem.} For $d\ge2$ even, and a continuous map $f : \Delta_{d+2} \to \IR^d$ there are two disjoint faces $\sigma_1,\sigma_2$ of $\Delta_{d+2}$ of dimension at most $ \frac{d}{2} $ in $\Delta_{d+2}$ such that $f(\sigma_1) \cap f(\sigma_2) \neq \emptyset$.
\end{compactitem}
}
\noindent
The generalized van Kampen--Flores theorems of Sarkaria \cite{sarkaria:generalized_van-kampen_flores} and Volovikov \cite{volovikov:van-kampen_flores} extend this by providing conditions that guarantee 
the existence of multiple overlaps. 
These theorems are classically obtained via topological methods that are considerably more involved
than those needed for to the proof of topological Tverberg theorem in cases of primes or
of prime powers.

Can we put any other restrictions on a Tverberg partition?
B\'ar\'any, F\"uredi \& Lov\'asz \cite{BFL} in 1989 realized a need for 
``a colored version of Tverberg's theorem.''
They proved the first instance of a such a result for three triangles
in the plane. 
Extending these ideas, B\'{a}r\'{a}ny \& Larman
\cite{barany1992colored} in 1992 formulated the following general problem.
{\em
\begin{compactitem}[\qquad]
\item  {\sc The colored Tverberg problem.} Determine the smallest natural number $n=n(d,r)$ such that for every collection
$\mathcal{C} =C_{0}\uplus \dots\uplus C_{d}$ of $n$ points in $\IR^d$, where each ``color class'' $C_i$
satisfies $|C_i|\ge r$, there are $r$
disjoint subcollections $F_{1},\dots,F_{r}$ of~$\mathcal{C}$ such that
\begin{compactenum}[\rm \quad(A)]
\item $|C_{i}\cap F_{j}|\le 1$ for every $i\in\{0,\dots,d\},\,j\in \{1,\dots,r\}$, and
\item $\mathrm{conv\,}(F_1)\cap\dots\cap\mathrm{conv\,}(F_r)\neq\emptyset$
\end{compactenum}
\end{compactitem}
}
\noindent
They proved that $n(1,r)=2r$, $n(2,r)=3r$, presented a proof by Lov\'asz for $n(d,2)=2(d+1)$, and conjectured that $n(d,r)=r(d+1)$.
In the same year \v{Z}ivaljevi\'c \& Vre\'cica \cite{vzivaljevic1992colored} formulated the following modified colored Tverberg problem:
{\em
\begin{compactitem}[\qquad]
\item  {\sc Modified colored Tverberg problem.}
Determine the smallest number $t=tt(d,r)$ such that for every
simplex $\Delta$ with $(d+1)$-colored vertex set
$\mathcal{C} =C_{0}\uplus \dots\uplus C_{d}$,  with $|C_{i}|\ge t$ for all~$i$,
and for every continuous map $f:\Delta\rightarrow\mathbb{R}^d$,
there are $r$ disjoint faces $\sigma_{1},\dots,\sigma_{r}$ of $\Delta$ satisfying
\begin{compactenum}[\rm \quad(A)]
\item $|C_{i}\cap \sigma_{j}|\le 1$ for every $i\in \{0,\dots,d\},\,j\in\{1,\dots,r\}$, and
\item $f(\sigma_1)\cap\dots\cap f(\sigma_r)\neq\emptyset$.
\end{compactenum}
\end{compactitem}
}
\noindent
\v{Z}ivaljevi\'c and Vre\'cica obtained the upper bound for the function $tt(d,r)\leq 2r-1$ in the case when $r$ is a prime that yield upper bound $tt(d,r)\leq 4r-3$ for any $r$.
 (\v{Z}ivaljevi\'c later extended this to prime powers $r$.)
These results were obtained by proving the nonexistence of  an $\mathfrak{S}_r$-equivariant map from the join
$\Delta_{r,2r-1}^{*(d+1)}$ into the sphere $S(W_r^{\oplus (d+1)})$ when $r$ is a prime power.
Here $\Delta_{r,2r-1}$ denotes the $r\times (2r-1)$ chessboard complex. 
The symmetric group  $\mathfrak{S}_r$ acts on $\Delta_{r,2r-1}$ by permuting the rows.  
The connectivity of the chessboard $\Delta_{r,2r-1}$ plays the central role in the proof of the nonexistence of an $\mathfrak{S}_r$-equivariant map $\Delta_{r,2r-1}^{*(d+1)}\to S(W_r^{\oplus (d+1)})$.
Note, however, that the upper bound of \v Zivaljevi\'c and Vre\'cica on the function $tt(d,r)$ does not provide any information about the function $n(d,r)$ of B\'{a}r\'{a}ny and Larman.

Only recently Blagojevi\'c,  Matschke \& Ziegler \cite{blagojevic2009optimal}, as a consequence of their ``optimal colored Tverberg theorem'', established
 that $n(d,r)=r(d+1)$ when $r+1$ is a prime and $n(d,r)\leq 2(r-1)(d+1)+1$ when $r$ is a prime.
Furthermore, they established $tt(d,r)\leq 2r-2$ for any~$r$.
See \cite{Z123e} for an exposition.
 
Can we ask for even more? What about preimages of the intersection point of a Tverberg partition? Can we say something about the barycentric coordinates of these preimages? This question was addressed by Sober\'on \cite{soberon:equal_coefficients} in 2013.
This result is discussed in more detail in Section \ref{sec:Soberon}.

\section{The topological Tverberg theorem with a constraining function}\label{sec:constraining_function}
 
The original version of Tverberg's theorem \cite{tverberg:generalisation_radon} asserts that any $(r-1)(d+1)+1$ points in $\IR^d$ can be partitioned into $r$ sets whose convex hulls intersect. 
As discussed, this can be phrased in terms of affine maps: Set $N := (r-1)(d+1)$ and denote by $\Delta_N$ the $N$-dimensional simplex;
then Tverberg's theorem says that for every affine map $f : \Delta_N \to \IR^d$ there are $r$ pairwise disjoint faces 
$\sigma_1, \dots, \sigma_r$ of~$\Delta_N$ such that $f(\sigma_1) \cap \dots \cap f(\sigma_r) \neq \emptyset$. 
The set of faces $\set{\sigma_1, \dots, \sigma_r}$ is called a \emph{Tverberg $r$-partition for $f$}, 
or simply a \emph{Tverberg partition} if it is clear which $f$ and~$r$ we refer to. 
Points $x_i \in \sigma_i$ for $i = 1, \dots, r$ with $f(x_1) = \dots = f(x_r)$ are \emph{points of Tverberg coincidence for $f$}. A first extension of this theorem to continuous maps was proved 
by Bárány, Shlosman \& Sz\H{u}cs \cite{barany_schlosman_szucs:toplogical_tverberg} for a prime number $r$ of intersecting faces. 
This was later generalized to prime powers $r$ by \"Ozaydin \cite{oezaydin:equivariant}. 

\begin{theorem} [(Topological Tverberg: Bárány, Shlosman \& Sz\H{u}cs \cite{barany_schlosman_szucs:toplogical_tverberg}, \"Ozaydin \cite{oezaydin:equivariant})]\label{thm:top_tverberg_thm}
  Let $r\ge2$ be a prime power, $d\ge1$, and $N \ge (r-1)(d+1)$. Then for every continuous map $f : \Delta_N \to \IR^d$ there are $r$ pairwise disjoint faces $\sigma_1, \dots, \sigma_r$ of $\Delta_N$ such that $f(\sigma_1) \cap \dots \cap f(\sigma_r) \neq \emptyset$.
\end{theorem}

We will consider an additional constraining function on $\Delta_N$, and ask that the points of coincidence of~$f$ lie in one fiber of this function. The proof of the following lemma shows that the topological Tverberg theorem itself yields that this can be achieved if we are are willing to increase the size of the original set of points in~$\IR^d$; for every continuous constraint we 
must provide $r-1$ additional points in order to obtain that there is a Tverberg partition equalizing the constraint.  

\begin{lemma}[(Key lemma \#1)] \label{tverberg_lemma}%
  Let $r\ge2$ be a prime power, $d\ge1$, and $c\ge0$. Let $N\ge N_c := (r-1)(d+1+c)$ and let $f : \Delta_N \to \IR^d$ and $g : \Delta_N \to \IR^c$ be continuous. Then there are $r$ points $x_i \in \sigma_i$, where $\sigma_1, \dots, \sigma_r$ are pairwise disjoint faces of~$\Delta_N$ with $g(x_1) = \dots = g(x_r)$ and $f(x_1) = \dots = f(x_r)$.  
\end{lemma}

\begin{proof} 
  Apply the topological Tverberg theorem  
  to the continuous  
  map $\Delta_N \to \IR^{d+c}$ given by $x \mapsto (f(x), g(x))$. 
\end{proof}

In this lemma, $c$ is the number of additional constraints; 
thus the special case $c=0$ is the topological Tverberg theorem.  
Suitable choices of constraining functions will enable us to obtain from this 
the existence of various kinds of special Tverberg partitions.

\begin{remark} \label{remark:affine}
  Note that Lemma \ref{tverberg_lemma} remains true for $r$ an arbitrary positive integer in the setting 
  where the map~$f$ as well as the constraint function $g$ are affine.
\end{remark} 

\section{Tverberg unavoidable subcomplexes}\label{sec:T-unavoidable-subcomplexes}

\begin{definition}[(Tverberg unavoidable subcomplexes)]
Let $r \ge 2$, $d \ge 1, N \ge r-1$ be integers and 
$f : \Delta_N \to \IR^d$ a continuous map with at least one Tverberg $r$-partition.
Then a subcomplex $\Sigma \subseteq \Delta_N$ is \emph{Tverberg unavoidable} 
if for every Tverberg partition $\set{\sigma_1, \dots, \sigma_r}$ 
for~$f$ there is at least one face $\sigma_j$ that lies in~$\Sigma$. 
\end{definition}

According to this definition, whether a subcomplex is Tverberg unavoidable depends on the parameters $r$, $d$, and~$N$, 
but also on the map $f$. 
However, we will only be interested in subcomplexes that are large enough to be unavoidable for \emph{any} map~$f$. 

\begin{lemma}[(Key examples)] \label{example:unavoidable} %
    Let $d\ge1$, $r\ge2$, and $N\ge r-1$, and assume that the continuous map $f : \Delta_{N} \to \IR^d$
    has a Tverberg $r$-partition. 
 \begin{compactenum}[\rm(i)]
  \item The induced subcomplex (simplex) $\Delta_{N-(r-1)}$ on $N-r+2$ vertices of~$\Delta_N$ is Tverberg unavoidable.\label{example:original_simplex}%

 \item For any set $S$ of at most $2r-1$ vertices in $\Delta_N$ the subcomplex of faces with at most one vertex in $S$ is Tverberg unavoidable.\label{example:colored}%
 
  \item If $k$ is an integer such that $r(k+2)> N+1$, then the $k$-skeleton $\Delta_N^{(k)}$ of $\Delta_N$ is Tverberg unavoidable.\label{example:bounded_dimensions}%
 
  \item If $k\ge0$ and $s$ are integers such that $r(k+1)+s> N+1$ with $0\le s \le r$, then the
     subcomplex $\Delta_N^{(k-1)}\cup\Delta_{N-(r-s)}^{(k)}$ of~$\Delta_N$ is Tverberg unavoidable.\label{example:bounded_dimensions_non_uniform}%
  \end{compactenum}
\end{lemma}

\begin{proof} All these are easy consequences of the pigeonhole principle:
    \begin{compactenum}[(i)]
        \item The simplex $\Delta_{N-(r-1)}$ contains all but $r-1$ of the vertices of $\Delta_N$, so
        for any Tverberg partition $\sigma_1, \dots, \sigma_r$ 
        at most $r-1$ of the faces $\sigma_i$ can have a vertex outside of $\Delta_{N-(r-1)}$, 
        so at least one face $\sigma_j$ has all vertices in $\Delta_{N-(r-1)}$. 
       
        \item If all faces $\sigma_1, \dots, \sigma_r$ had at least two vertices in $S$, then 
        we would have $\abs{S} \ge \sum_{i=1}^r \abs{\sigma_i \cap S} \ge 2r$.
        
     \item If all faces $\sigma_1, \dots, \sigma_r$ of a Tverberg partition had dimension at least $k+1$, 
        then this would involve at least $r(k+2)$ vertices.
        
    \item
         If none of the faces $\sigma_1,\dots,\sigma_r$ lies in $\Delta_N^{(k-1)}\cup\Delta_{N-(r-s)}^{(k)}$,
          then they all have dimension at least $k$, and since they are disjoint only $r-s$ of them can involve one of the
          last $r-s$ vertices, so $s$ of them must have dimension at least $k+1$.
          For this $r(k+1)+s$ vertices are needed.\vspace{-4mm} 
    \end{compactenum}
\end{proof}

Parts (\ref{example:original_simplex}), (\ref{example:colored}), and (\ref{example:bounded_dimensions}) 
of the lemma follow from the following more general statement: 
If $S$ is a set of at most $(s+1)r-1$ vertices, $s \ge 0$, then the subcomplex of faces with at most~$s$ vertices 
in~$S$ is Tverberg unavoidable. (Again, this immediately follows from the pigeonhole principle.)
For $s=0$ this gives that for a set $S$ of at most $r-1$ vertices there is a face not intersecting $S$ (having $0$ vertices in common with $S$) in every Tverberg $r$-partition; this is part~(\ref{example:original_simplex}).
For $s=1$ this yields part~(\ref{example:colored}): Here $S$ has size at most $2r-1$ and there is a face with 
at most one vertex in $S$.
For $s$ large enough that $S$ may contain all $N+1$ vertices, this gives a bound on the skeleton. 
Namely, if $(s+1)r-1 \ge N+1$ then there is a face with at most $s$ vertices, i.e., a face in the $(s-1)$-skeleton: 
This yields part~(\ref{example:bounded_dimensions}) for $k=s-1$. 
We thank the anonymous referee for this generalization.

If $r$ is a prime power and $N\ge N_0=(r-1)(d+1)$, then the topological Tverberg theorem  
guarantees the existence of a Tverberg $r$-partition for \emph{any} continuous $f:\Delta_N\to\IR^d$.
We now show that for $N\ge N_1 = (r-1)(d+2)$, every subcomplex $\Sigma\subseteq\Delta_N$ that is Tverberg unavoidable 
for $f$ necessarily contains \emph{all the faces} of some Tverberg $r$-partition.
 
\begin{lemma}[(Key lemma \#2)] \label{lemma:unavoidable}
  Let $r\ge2$ be a prime power, $d\ge1$, and $N\ge N_1 = \allowbreak (r-1)(d+2)$.
  Assume that $f : \Delta_N \to \IR^d$ is continuous and that
  the subcomplex $\Sigma \subseteq \Delta_N$ is Tverberg unavoidable for~$f$.  
  Then there are $r$ pairwise disjoint faces $\sigma_1, \dots, \sigma_r$ of $\Delta_N$, all of~them contained in $\Sigma$,
  such that $f(\sigma_1) \cap \dots \cap f(\sigma_r) \neq \emptyset$. 
\end{lemma}

\begin{proof}
  Let $g : \Delta_N \to \IR$ assign to each point $x \in \Delta_N$ its distance to $\Sigma$. This map $g$ is continuous.
  We have $g(x) = 0$ if and only if $x$ belongs to $\Sigma$, since $\Sigma\subseteq\Delta_N$ is closed.
 
By Lemma \ref{tverberg_lemma} there are points $x_1, \dots, x_r$ 
in pairwise disjoint faces $\sigma_1, \dots, \sigma_r\subset\Delta_N$ 
with $f(x_1) =   \dots = f(x_r)$ and $g(x_1) =   \dots = g(x_r)$. 
We may assume that the $\sigma_i$ are inclusion-minimal with $x_i \in \sigma_i$, that is, 
$\sigma_i$ is the unique face with $x_i$ in its relative interior. 
Since $\Sigma$ is Tverberg unavoidable for~$f$, at least one of these faces, say $\sigma_j$, is contained in $\Sigma$. 
Thus $g(x_j) = 0$, which implies that $g(x_i) = 0$ for all $i = 1, \dots, r$. 
Thus all points $x_i$ lie in $\Sigma$ and since the $\sigma_i$ are inclusion-minimal 
and $g$ vanishes at all points in the relative interior of a face if and only if it vanishes at some such point, 
the $\sigma_i$ belong to~$\Sigma$.%
\end{proof}
 
With Lemma \ref{example:unavoidable}(\ref{example:original_simplex}), this shows that
the topological Tverberg theorem for maps to $\IR^{d+1}$
immediately yields the lower-dimensional version for maps to $\IR^d$. 
Thus for some $r$ (not necessarily a prime power) it suffices to show the topological
Tverberg theorem for arbitrarily large dimensions.
(See de Longueville \cite[Prop. 2.5]{delongueville2001notes} for this observation: 
his proof uses an extended target space and the observation that the induced subcomplex of $N-r+2$ vertices is unavoidable.) 
Lemma~\ref{example:unavoidable}(\ref{example:colored}) will in the next section get us colored versions of the topological Tverberg theorem.
From the parts (\ref{example:bounded_dimensions}) and (\ref{example:bounded_dimensions_non_uniform}) 
of Lemma~\ref{example:unavoidable} we will in Section~\ref{sec:prescribing_dimensions}
derive versions of the topological Tverberg theorem with a bound on the dimension of the intersecting faces.

The following is a straightforward extension of Lemma~\ref{lemma:unavoidable}.

\begin{theorem} \label{theorem:multiple_unavoidable}
  Let $r\ge2$ be a prime power, $d\ge1$, $c\ge1$, and $N\ge N_c = (r-1)(d+1+c)$.\break%
  Let $f : \Delta_N \to \IR^d$ be continuous and let $\Sigma_1,\Sigma_2,\dots,\Sigma_c \subseteq \Delta_N$ be Tverberg unavoidable subcomplexes for~$f$. 
  Then there are $r$ pairwise disjoint faces $\sigma_1, \dots, \sigma_r$ in $\Sigma_1\cap\dots\cap \Sigma_c$ such that $f(\sigma_1) \cap \dots \cap f(\sigma_r) \neq \emptyset$. 
\end{theorem}

\begin{proof}
    We may assume that $N=N_c$.
  Let $g_i : \Delta_{N_c} \to \IR$ assign to $x \in \Delta_N$ the distance to the subcomplex $\Sigma_i$, 
and consider $g : \Delta_{N_c} \to \IR^c$, $x \mapsto (g_1(x), \dots, g_c(x))$. 
Now use Lemma \ref{tverberg_lemma} and that the~$\Sigma_i$ are Tverberg unavoidable.
\end{proof}
 
\begin{question}
  Does Theorem \ref{theorem:multiple_unavoidable} remain true for $r$ an arbitrary positive integer and $f$ affine?
\end{question}

\section{Weak colored versions of the Tverberg theorem}\label{sec:colored}

An interesting way to constrain Tverberg partitions is to color the vertices of the simplex $\Delta_N$
and to require that the faces in the Tverberg partition have no two vertices of the same color:
B\'ar\'any \& Larman \cite{barany1992colored} asked for the minimal $N$ such that for any affine map and for any
coloring of the $N+1$ vertices of~$\Delta_N$ by $d+1$ colors, where each color class should have size at least~$r$,
a Tverberg partition of $r$ faces, each without repeated colors, exists. 

Suppose that the vertices of $\Delta_N$ are partitioned into color classes.  
Denote by $R\subseteq\Delta_N$ the \emph{rainbow complex}, 
that is, the subcomplex of faces that have at most one vertex of each color class. These faces are called \emph{rainbow faces}. 
Lemma \ref{tverberg_lemma} implies that for any coloring of the vertices of $\Delta_N$ with $N\ge N_1 = (r-1)(d+2)$ 
and for any continuous map $f : \Delta_N \to \IR^d$ there is a Tverberg partition 
such that the points of Tverberg coincidence for $f$ have the same distance to the rainbow complex. 
We will determine some conditions that a coloring must fulfil such that the rainbow complex is Tverberg 
unavoidable or is an intersection of few Tverberg unavoidable subcomplexes.

\begin{lemma}  
  Let $r\ge2$ be a prime power, $d \ge 1$, and $N\ge N_1 = (r-1)(d+2)$. Let $S \subseteq \Delta_N$ be any set of at most $2r-1$ vertices, 
  and let $f : \Delta_N \to \IR^d$ be any continuous function. 
  Then there are $r$ pairwise disjoint faces $\sigma_1, \dots, \sigma_r$ of~$\Delta_N$ with $\abs{\sigma_i \cap S} \le 1$ such that 
  $f(\sigma_1) \cap \cdots\allowbreak \cap f(\sigma_r) \neq \emptyset$. 
\end{lemma}

\begin{proof}
  The subcomplex $\Sigma$ of faces with at most one vertex in $S$ is Tverberg unavoidable by Lemma \ref{example:unavoidable}(\ref{example:colored}). Thus $\Sigma$ contains a Tverberg partition.
\end{proof}
 
This is a colored version of the topological Tverberg theorem, where the vertices in $S$ are in one color class and all other vertices are colored with distinct colors. 
If we assume that $\abs{S} \ge r$ then we can also state this theorem with equalities $\abs{\sigma_i \cap S} = 1$ by adding a point in $S$ to every face $\sigma_i$ that is disjoint from $S$. 

The following ``colored Radon theorem'' (a restatement of the Borsuk--Ulam theorem) is a direct consequence for $r = 2$.

\begin{corollary}[(Colored Radon: Vre\'cica \& \v{Z}ivaljevi\'c {\cite[Cor.~7]{vrecica_zivaljevic:chessboard_complexes}})]\label{cor:colored_Radon} 
  Let $d \ge 1$,  let  the map $f : \Delta_{d+2} \to \IR^d$ be continuous and let $S \subseteq \Delta_{d+2}$ be a set of three vertices. Then there are disjoint faces $\sigma_1, \sigma_2$ in $\Delta_{d+2}$ with $\abs{\sigma_1 \cap S} \le 1$ and $\abs{\sigma_2 \cap S} \le 1$ such that $f(\sigma_1) \cap f(\sigma_2) \neq \emptyset$.
\end{corollary}

\begin{theorem}[(Weak colored Tverberg)] \label{theorem:colored_tverberg}
  Let $r\ge2$ be a prime power, $d \ge 1$, $N\ge N_{d+1} = (r-1)(2d+2)$ and let $f : \Delta_N \to \IR^d$ be continuous. 
  If the vertices of $\Delta_N$ are colored by $d+1$ colors, where each color class has cardinality at most $2r-1$,
  then there are $r$ pairwise disjoint rainbow faces $\sigma_1, \dots, \sigma_r$ of~$\Delta_N$ such that $f(\sigma_1) \cap \dots \cap f(\sigma_r) \neq \emptyset$. 
\end{theorem}
  
\begin{proof} 
  For each fixed color $i$, the subcomplex $\Sigma_i$ of faces that have at most one vertex of color $i$ is Tverberg unavoidable 
  by Lemma \ref{example:unavoidable}(\ref{example:colored}). The complex of rainbow faces is $\Sigma_1\cap\cdots\allowbreak \cap \Sigma_{d+1}$. 
  Now use Theorem~\ref{theorem:multiple_unavoidable}.
\end{proof}

Note that in Theorem~\ref{theorem:colored_tverberg} the fact that all vertices can be colored with $d+1$ colors of size at most $2r-1$
implies that $N+1 \le (2r-1)(d+1)$. The theorem is ``weak'' as it needs a large number of points/vertices to reach its
conclusion, namely $N+1\ge N_{d+1}+1=(r-1)(2d+2)+1$, while the optimal result requires only $N+1\ge N_0+1=(r-1)(d+1)+1$ of them, as in
Theorem~\ref{theorem:optimal} below.
The special case of Theorem~\ref{theorem:colored_tverberg} when all color classes have the same  
cardinality $2r-1$, and thus $N+1 = (d+1)(2r-1)$, is the colored Tverberg theorem of 
\v{Z}ivaljevi\'c \& Vre\'cica \cite{vzivaljevic1992colored}. As we do \emph{not} need to require all color
classes to have the same size (a simple observation that apparently was first made in \cite{blagojevic2009optimal}), 
we need $d$ points/vertices less to force a colored Tverberg partition.

Theorem~\ref{theorem:colored_tverberg} leaves some flexibility in the choice of color classes: For example, we could consider $d$ colors of cardinality $2r-2$ and one color class of size $2r-1$. Instead of shrinking the size of color classes we can also allow fewer color classes. This gives a colored Tverberg theorem ``of type B'' 
(in terminology of Vre\'cica \& \v{Z}ivaljevi\'c introduced in \cite{vrecica1994new}), 
that is, fewer than $d+1$ colors are possible. 

We proceed using the method provided in Sections~\ref{sec:constraining_function} and \ref{sec:T-unavoidable-subcomplexes}, 
and thus obtain the following result, which 
in the special case where all color classes have the same size is the main result of~\cite{vrecica1994new}.
It also implies Theorem~\ref{theorem:colored_tverberg}.

\begin{theorem}
  Let $r\ge2$ be a prime power, $d \ge 1$, $c \ge \lceil \frac{r-1}{r}d \rceil + 1$, and $N\ge N_c = \break(r-1)(d+1+c)$.
  Let $f : \Delta_N \to \IR^d$ be continuous. If the vertices of $\Delta_N$ are divided into $c$ color classes,
  each of them of cardinality at most $2r-1$, then there are $r$ pairwise disjoint rainbow faces $\sigma_1, \dots, \sigma_r$ of~$\Delta_N$ such that $f(\sigma_1) \cap \dots \cap f(\sigma_r) \neq \emptyset$. 
\end{theorem}

\begin{proof}
  Again we may assume that $N=N_c$.
  We need that all $N_c+1$ vertices can be colored by $c$ colors of cardinality at most $2r-1$ to use 
  Theorem~\ref{theorem:multiple_unavoidable}, that is, we need $c(2r-1)\ge N_c+1$, 
  which is equivalent to $c \ge \lceil \frac{r-1}{r}d \rceil + 1$.
\end{proof}

\section{Prescribing dimensions for Tverberg simplices}\label{sec:prescribing_dimensions}

If $N$ is sufficiently large, can we set an upper bound $\dim(\sigma_i)\le k$ for the
dimensions of the faces in a Tverberg $r$-partition for $f:\Delta_N\to\IR^d$? 
For the case $r = 2$ such a result is due to 
van Kampen \cite{van1933komplexe} and independently Flores~\cite{flores:selbstverschlungen}.%

\begin{theorem}  [(van Kampen--Flores)]\label{thm:vanKampen-Flores}
  Let $d\ge2$ be even. Then for every continuous map $f : \Delta_{d+2} \to \IR^d$ there are two disjoint faces $\sigma_1,\sigma_2\subset\Delta_{d+2}$ of dimension at most $ \frac{d}{2} $ in $\Delta_{d+2}$ with $f(\sigma_1) \cap f(\sigma_2) \neq \emptyset$.
\end{theorem}

This was generalized to a prime number $r$ of faces by Sarkaria \cite{sarkaria:generalized_van-kampen_flores}
and later to prime powers $r$ by Volovikov~\cite{volovikov:van-kampen_flores}. In fact, Volovikov's result holds for maps to manifolds that induce a trivial homomorphism on cohomology in dimension $k$, where $k$ is the desired bound on the dimension of faces. We restate Volovikov's result here for the case that the target space is Euclidean space. A collection of sets $S_1, \dots, S_r$ is called \emph{$j$-wise disjoint} if the intersection of any $j$ of these sets is empty. 

\begin{theorem}[(Generalized van Kampen--Flores: Sarkaria \cite{sarkaria:generalized_van-kampen_flores}, Volovikov \cite{volovikov:van-kampen_flores})]%
    \label{theorem:generalized_van-kampen_flores}%
  Let $r\ge2$ be a prime power, $2\le j\le r$, $d\ge1$, and $k<d$ such that
  there is an integer $m\ge0$ that satisfies
  \begin{equation}\label{eqn:SarkariaVolovikov-condition}
    (r-1)(m+1)+r(k+1)\ \ \ge\ \ (N+1)(j-1)\ \ >\ \ (r-1)(m+d+2).  
  \end{equation}
  Then for every continuous map $f : \Delta_N \to \IR^d$ there are $r$ $j$-wise disjoint faces 
  $\sigma_1, \dots, \sigma_r$ of~$\Delta_N$ with $\dim\sigma_i \le k$ for $1\le i\le r$,
   such that $f(\sigma_1) \cap \dots \cap f(\sigma_r) \neq \emptyset$.
\end{theorem}

Let us discuss which of the conditions posed by (\ref{eqn:SarkariaVolovikov-condition}) 
are necessary. First, the left-hand-side has to be strictly larger than the 
right-hand-side, which yields the condition $(m+1)(r-1)+r(k+1) - (m+d+2)(r-1)>0$, that is,
\begin{equation}\label{eqn:condition1}
    k\ge \tfrac{r-1}{r}d. 
\end{equation}
This lower bound on $k$ is necessary, as we see by looking
at a generic affine map $f$, which does not have the desired Tverberg $r$-partition unless
the sum of the codimensions of the $\sigma_i$ is at most $d$, that is, $r(d-k)\le d$.  

If the second inequality in (\ref{eqn:SarkariaVolovikov-condition}) is satisfied for
some $m\ge0$, then it is in particular satisfied for $m = 0$, which gives $(N+1)(j-1) > (r-1)(d+2)$, that is,
\begin{equation}\label{eqn:condition2}
    N+1>\tfrac{r-1}{j-1}(d+2). 
\end{equation} 
It is not clear whether this lower bound on $N$
is necessary in general; 
\[
    N+1 > \lfloor\tfrac{r-1}{j-1}\rfloor(d+2)
\]
is necessary for $k<d$,
as one can see from an affine map $\Delta_N\rightarrow\Delta_d$ that maps at most $\lfloor\frac{r-1}{j-1}\rfloor$
vertices of $\Delta_N$ to each of the vertices and to the barycenter of a $d$-simplex. 
This example is suggested by Sarkaria in~\cite{sarkaria:generalized_van-kampen_flores}.
Note that for $j=2$ the lower bound of (\ref{eqn:condition2}) reads $N+1>(r-1)(d+2)$,
which \emph{is} optimal, despite a mistaken claim  
in \cite[Thm.~1.5 and the sentence after this]{sarkaria:generalized_van-kampen_flores} 
that the bound $N\ge r(s+1)-2$ is optimal, where $s=k+1$ in Sarkaria's notation.
In the example he gives, the two bounds coincide.

Even if both conditions (\ref{eqn:condition1}) and (\ref{eqn:condition2}) hold
the integer $m$ that should satisfy (\ref{eqn:SarkariaVolovikov-condition})
may not exist. This requirement is non-trivial, see Example \ref{example:integrality} below. It is 
not necessary, as we shall see.
 
We will get our strengthened version of Theorem~\ref{theorem:generalized_van-kampen_flores} as a direct consequence of the topological Tverberg theorem.
For this we first establish the case $j = 2$ as a corollary of Lemma~\ref{lemma:unavoidable}. 
 
\begin{theorem}  \label{theorem:bounded_dimensions}%
  Let $r\ge2$ be a prime power, $d \ge 1$, $N\ge N_1 = (r-1)(d+2)$, 
  and $k \ge \lceil\frac{r-1}{r}d\rceil$. 
  then for every continuous map $f : \Delta_N \to \IR^d$ there are $r$ pairwise disjoint faces $\sigma_1, \dots, \sigma_r$ 
  of~$\Delta_N$, with $\dim\sigma_i \le k$ for $1\le i\le r$, 
  such that $f(\sigma_1) \cap \dots \cap f(\sigma_r) \neq \emptyset$.
\end{theorem}

\begin{proof}
    It is sufficient to prove this for $N=N_1$.
  The $k$-skeleton $\Delta_{N_1}^{(k)}$ of $\Delta_{N_1}$ is Tverberg unavoidable 
  by Lemma~\ref{example:unavoidable}(\ref{example:bounded_dimensions}).
\end{proof}

\begin{example}
For $d = r = 3$ and $f$ an affine map, this theorem asserts that given eleven points in $\IR^3$, 
one can find three pairwise disjoint sets of three points whose convex hulls intersect.  
Ten points are not sufficient for this, as by the discussion above one needs more than $(r-1)(d+2) = 10$ points. 
(This solves a problem discussed by Matou\v{s}ek in \cite[Example 6.7.4]{matousek:borsuk-ulam2}.)
\end{example}

We now obtain our strengthening of Theorem~\ref{theorem:generalized_van-kampen_flores}  
as a corollary of Theorem~\ref{theorem:bounded_dimensions}, the special case $j=2$.

\begin{theorem}[(Generalized van Kampen--Flores, sharpened)] \label{theorem:j-wise_disj_bounded_dimensions}
  Let $r\ge2$ be a prime power, $2\le j\le r$, $d\ge1$, and $k\le N$ such that
  \begin{equation}\label{eqn:SarkariaVolovikov-reloaded}
      k\ge \tfrac{r-1}{r}d    
     \quad\textrm{and}\quad
      N+1>\tfrac{r-1}{j-1}(d+2).   
  \end{equation}
  Then for every continuous map $f : \Delta_N \to \IR^d$ there are $r$ $j$-wise disjoint faces 
  $\sigma_1, \dots, \sigma_r$ of~$\Delta_N$, with $\dim\sigma_i \le k$ for $1\le i\le r$,
   such that $f(\sigma_1) \cap \dots \cap f(\sigma_r) \neq \emptyset$.
\end{theorem}

\begin{proof}
	Let $N' := (N+1)(j-1)-1$, let $p$ be the natural simplicial projection $\Delta_{N'} \cong \Delta_N^{*(j-1)} \to \Delta_N$ 
	that maps each of the $j-1$ copies of a vertex $v \in \Delta_N$ in the join $\Delta_N^{*(j-1)}$ to the vertex $v$,
	and set $f' := f \circ p : \Delta_{N'} \to \IR^d$. 
	
	We have $N' \ge (r-1)(d+2)$. Thus by Theorem \ref{theorem:bounded_dimensions} 
	there are pairwise disjoint faces $\sigma'_1, \dots, \sigma'_r \subseteq \Delta_{N'}$, 
	with  $\dim \sigma_i \le k$ for all $i$, 
	such that $f'(\sigma'_1) \cap \dots \cap f'(\sigma'_r) \neq \emptyset$. 
	By definition of~$f'$ this is equivalent to $f(p(\sigma'_1)) \cap \dots \cap f(p(\sigma'_r)) \neq \emptyset$. 
	Now, let $\sigma_1 = p(\sigma'_1), \dots, \sigma_r = p(\sigma'_r)$. 
	The faces $p(\sigma'_i)$ of $\Delta_N$ still satisfy the dimension bound $\dim p(\sigma'_i) \le k$, and they are $j$-wise disjoint. 
\end{proof}

\begin{example}\label{example:integrality}
	To see that Theorem~\ref{theorem:j-wise_disj_bounded_dimensions} is stronger than Theorem \ref{theorem:generalized_van-kampen_flores}, let $d = j = r = 3$ and $k = 2$. 
	Then the prerequisites of Theorem \ref{theorem:j-wise_disj_bounded_dimensions} 
	are satisfied for $N = 5$. Thus for any continuous map $f : \Delta_5 \to \IR^3$ 
	there are three $3$-wise disjoint faces of dimension at most~$2$ whose images intersect.
	
	However, inequality (\ref{eqn:SarkariaVolovikov-condition}) of Theorem \ref{theorem:generalized_van-kampen_flores} 
	asks that $(m+1)\cdot 2 + 3\cdot 3 \ge (N+1) \cdot 2 > \break (m+5)\cdot 2$, that is, $2m+11 \ge 2N+2 > 2m +10$. 
	Such an integer $m$ exists for no $N$, as $2N+2$ is even.
\end{example}

One can ask for a further strengthening of Theorem \ref{theorem:j-wise_disj_bounded_dimensions} where       
we would not put the same dimension bound on all simplices~$\sigma_i$. From Lemma ~\ref{example:unavoidable}(\ref{example:bounded_dimensions_non_uniform})
we get the following.

\begin{theorem}[(Generalized van Kampen--Flores, sharpened further)] \label{theorem:bounded_dimensions_non-uniform}
  Let $r\ge2$ be a prime power, $d \ge 1$, $N\ge (r-1)(d+2)$, and 
  \[
  r(k+1)+s> N + 1\quad \text{ for integers }k\ge0\text{ and }0\le s < r.
  \] 
  then for every continuous map $f : \Delta_N \to \IR^d$ there are $r$ pairwise disjoint faces 
  $\sigma_1,\dots,\sigma_r$ of $\Delta_N$ such that 
  $f(\sigma_1) \cap \dots \cap f(\sigma_r) \neq \emptyset$, $\dim(\sigma_i)\le k$ for all $i$, and 
  the number $\ell$ of simplices $\sigma_i$ with $\dim\sigma_i=k$ satisfies $\ell(k+1)\le N-(r-s)+1$.
\end{theorem}

\begin{proof} 
  The complex $\Delta_N^{(k-1)}\cup\Delta_{N-(r-s)}^{(k)}$ is Tverberg unavoidable by Lemma~\ref{example:unavoidable}(\ref{example:bounded_dimensions_non_uniform}).
  Thus there is a Tverberg partition $\sigma_1,\dots,\sigma_r$ with all simplices in the unavoidable subcomplex, so we get $\dim\sigma_i\le k$.
  Moreover, the simplices $\sigma_i$ altogether can take up only
  $N+1\le r(k+1)+s$ vertices. 
\end{proof}

We leave it to the reader to state and prove a $j$-wise disjoint version of Theorem~\ref{theorem:bounded_dimensions_non-uniform}.
On the other hand even for $j=2$ we do not, up to now, seem to get the full result that one could hope for: 

\begin{conjecture} \label{conjecture:bounded_dimensions_non-uniform}
  Let $r\ge2$ be a prime power, $d \ge 1$, $N\ge (r-1)(d+2)$, and 
  \[
  r(k+1)+s> N + 1\quad \text{ for integers }k\ge0\text{ and }0\le s<r.
  \] 
  then for every continuous map $f : \Delta_N \to \IR^d$ there are $r$ pairwise disjoint faces $\sigma_1,\dots,\sigma_r$ of~$\Delta_N$ such that 
  $f(\sigma_1) \cap \dots \cap f(\sigma_r) \neq \emptyset$, with 
  $\dim\sigma_i\le k+1$ for $1\le i\le s$  and  
  $\dim\sigma_i\le k$\break   for $s <  i\le r$.%
\end{conjecture}

More generally, Roland Bacher has asked on \texttt{mathoverflow} \cite{various:mathoverflow} which dimensions $d_i=\dim\sigma_i$
could be prescribed for a Tverberg partition if the number of points $N$ is sufficiently large.
We have already noted that a Tverberg $r$-partition in which the codimensions of the $\sigma_i$ add to more than $d$
will not exist for an affine general position map, so we need to assume that $\sum_i (d-d_i) \le d$.
Also arbitrarily large families of $N$ points on the moment curve, whose convex hulls are neighborly polytopes, show that we
cannot force that $\dim\sigma_i< \lfloor \frac d2\rfloor$ for any~$i$.

\begin{definition}[(admissible, Tverberg prescribable)]
For $d\ge1$ and $r\ge2$ an $r$-tuple $(d_1, \dots, d_r)$ of integers is \emph{admissible for $d$}
if  $\lfloor \frac{d}{2} \rfloor \le d_i \le d$ and $\sum_{i=1}^r (d-d_i) \le d$. An admissible $r$-tuple $(d_1, \dots, d_r)$ is \emph{Tverberg prescribable} if there is an $N$ such that for every continuous $f : \Delta_N \to \IR^d$ there is a Tverberg partition $\set{\sigma_1, \dots, \sigma_r}$ for $f$ with $\dim \sigma_i = d_i$.
\end{definition}

Theorem \ref{theorem:bounded_dimensions} shows that every admissible $r$-tuple of equal integers is Tverberg prescribable; see also Haase~\cite{haase:tverberg}. 
Also in the case $r = 2$ all admissible tuples $(d_1, d_2)$ are Tverberg prescribable. Indeed, 
for even $d$ this is given by the van Kampen--Flores Theorem \ref{thm:vanKampen-Flores}.
However, for odd $d$ we need a statement that is stronger than
what you get directly from Theorems~\ref{theorem:generalized_van-kampen_flores} or~\ref{theorem:bounded_dimensions}, 
namely, that there are two disjoint faces
of $\Delta_{d+2}$, both of them of dimension at most~$\lceil\frac d2\rceil$, whose images intersect.
The version that we need will be obtained from Theorem~\ref{theorem:bounded_dimensions_non-uniform}.

\begin{theorem}[(van Kampen--Flores, sharpened)]\label{thm:vanKampen-Flores_sharpened} 
	Let $d\ge1$.
	Then for every continuous map\break
	$f : \Delta_{d+2} \to \IR^d$ there are two disjoint faces $\sigma_1, \sigma_2$ of $\Delta_{d+2}$ such that 
	$\dim \sigma_1 = \lceil\frac d2\rceil$, $\dim \sigma_2 = \lfloor\frac d2\rfloor$, and $f(\sigma_1) \cap f(\sigma_2) \neq \emptyset$.
\end{theorem}

\begin{proof}
    It remains to settle the case when $d$ is odd, with $\dim \sigma_1 = \frac{d-1}{2}$ and $\dim \sigma_2 = \frac{d+1}{2}$.
    In terms of Theorem~\ref{theorem:bounded_dimensions_non-uniform}, in this situation we have $r=2$, $k=\frac{d+1}2$, $s=1$, and thus
    there is a Tverberg $2$-partition $\sigma_1,\sigma_2$ with $\dim \sigma_i\le\frac{d+1}2$,
    where at most $\ell\le\big\lfloor\frac{N-(r-s)+1}{k+1}\big\rfloor=\big\lfloor\frac{(d+2)-(2-1)+1}{\frac{d+1}2+1}\big\rfloor=1$ 
    of the $\sigma_i$ have dimension~$k=\frac{d+1}2$.
\end{proof}

For $d=3$, this says that for every continuous map $\Delta_5\rightarrow\IR^3$, the images of a triangle and an edge of~$\Delta_5$ intersect. This is equivalent to the Conway--Gordon--Sachs theorem \cite{conway-gordon-intrinsically-linked} \cite{sachs-intrinsically-linked} from graph theory, which says that the complete graph $K_6$ (that is,
the $1$-skeleton of $\Delta_5$) is ``intrinsically linked''. 

\begin{question}
	Is every admissible $r$-tuple Tverberg prescribable?
\end{question}

\section{{j}-wise disjoint Tverberg partitions}\label{sec:j-wise}

The following result is a version of Theorem~\ref{theorem:generalized_van-kampen_flores}, without a bound on the dimensions of
the simplices $\sigma_i$ in the Tverberg partition. We state it here, and give a simpler proof, since 
 our methods will also yield a new affine version of this, which holds for all $r\ge2$.

\begin{theorem}[($j$-wise disjoint topological Tverberg: Sarkaria \cite{sarkaria:generalized_van-kampen_flores}, Volovikov \cite{volovikov:van-kampen_flores})]\label{thm:j-wise}
	Let $r\ge2$ be a prime power, $d \ge 1$, $2 \le j \le r$, and 
\begin{equation}
     N+1 > \tfrac{r-1}{j-1}(d+1).
\end{equation}
Then for every continuous map $f : \Delta_N \to \IR^d$ there are $j$-wise disjoint faces $\sigma_1, \dots, \sigma_r$ of~$\Delta_N$ such that $f(\sigma_1) \cap \dots \cap f(\sigma_r) \neq \emptyset$.
\end{theorem}

\begin{proof}
      For this we repeat the proof of Theorem \ref{theorem:j-wise_disj_bounded_dimensions} and use the topological Tverberg theorem, Theorem~\ref{thm:top_tverberg_thm}, in place of Theorem~\ref{theorem:bounded_dimensions}.
	Let $N' := (N+1)(j-1)-1$, let $p$ be the natural simplicial projection $\Delta_{N'} \cong \Delta_N^{*(j-1)} \to \Delta_N$ 
	that maps each of the $j-1$ copies of a vertex $v \in \Delta_N$ in the join $\Delta_N^{*(j-1)}$ to the vertex $v$,
	and set $f' := f \circ p : \Delta_{N'} \to \IR^d$. 
	
	We have $N' \ge (r-1)(d+1)$. Thus by the topological Tverberg theorem, Theorem~\ref{thm:top_tverberg_thm},
	there are pairwise disjoint faces $\sigma'_1, \dots, \sigma'_r \subseteq \Delta_{N'}$
	such that $f'(\sigma'_1) \cap \dots \cap f'(\sigma'_r) \neq \emptyset$. 
	By definition of $f'$ this is equivalent to $f(p(\sigma'_1)) \cap \dots \cap f(p(\sigma'_r)) \neq \emptyset$. 
	Now, let $\sigma_1 = p(\sigma'_1), \dots, \sigma_r = p(\sigma'_r)$. 
	The faces $p(\sigma'_i)$ of $\Delta_N$ are $j$-wise disjoint. 
\end{proof}

\begin{theorem}[($j$-wise disjoint Tverberg)]
	Let $r \ge 2$, $d \ge 1$, $2 \le j \le r$, and 
\begin{equation}
     N+1 > \tfrac{r-1}{j-1}(d+1).
\end{equation}
Then for every affine map $f : \Delta_N \to \IR^d$ there are $j$-wise disjoint faces 
$\sigma_1, \dots, \sigma_r$ of~$\Delta_N$ such that $f(\sigma_1) \cap \dots \cap f(\sigma_r) \neq \emptyset$.
\end{theorem}

\begin{proof}
For affine maps $f : \Delta_N \to \IR^d$ we need not even assume that $r$ is a prime power---if we use
Tverberg's original theorem as the black box result. 
This is possible since the projection map
 $p : \Delta_N^{*(j-1)} \to \Delta_N$ 
as in the proof of Theorem \ref{thm:j-wise} is affine.
\end{proof}

\section{Tverberg partitions with equal barycentric coordinates}\label{sec:Soberon}
 
Let the vertices of $\Delta_N$ be partitioned into $\ell$ color classes $C_0, \dots, C_{\ell-1}$. Every point $x \in R$ in the 
corresponding rainbow complex $R$ has unique barycentric coordinates $x = \sum_{i=0}^{\ell-1} \alpha_iv_i$ with 
$0 \le \alpha_i \le 1$ and $v_i$ a vertex in the color class $C_i$ for $0\le i\le \ell-1$. 
We say that two points $x, y$ in the rainbow complex have \emph{equal barycentric coordinates} 
if $x = \sum_{i=0}^{\ell-1} \alpha_iv_i$ and $y = \sum_{i=0}^{\ell-1} \alpha_iw_i$, where $v_i$ and $w_i$ are vertices in the color class~$C_i$. 
The following theorem is a topological version of Sober\'on's \cite{soberon:equal_coefficients} 
``Tverberg's theorem with equal barycentric coordinates.''

\begin{theorem}[(Topological Tverberg with equal barycentric coordinates)] \label{theorem:equal_coeff}
Let $r\ge2$ be a prime power, $d \ge 1$, and $N = N_{(r-1)d} = r((r-1)d+1)-1$. Let $f : \Delta_N \to \IR^d$ be continuous. 
If the vertices of $\Delta_N$ are partitioned into $(r-1)d+1$ color classes of size $r$, 
then there are points $x_1, \dots, x_r$ with equal barycentric coordinates in $r$ pairwise disjoint rainbow faces 
$\sigma_1, \dots, \sigma_r$ of~$\Delta_N$ whose images intersect, with $f(x_1) = \dots = f(x_r)$.
\end{theorem}

\begin{proof}
  Let the color classes be $C_0, \dots, C_{(r-1)d}$.
  Every point $x \in \Delta_N$ is a unique convex combination $x = \sum \alpha_iv_i$ of the vertices of $\Delta_N$. 
  For $0\le k\le (r-1)d$ let $g_k : \Delta_N \to \IR$ be given by $\sum \alpha_iv_i \mapsto \sum_{v_i \in C_k} \alpha_i$. 
  Each $g_k$ is an affine function that is equal to~$1$ on the simplex $\conv(C_k)\subset\Delta_N$ with vertex set~$C_k$ 
  and $0$ on all other vertices of~$\Delta_N$. 
  
By Lemma \ref{tverberg_lemma} there are $x_1, \dots, x_r \in \Delta_N$ with $x_i \in \sigma_i$, 
where the $\sigma_i\subset\Delta_N$ are pairwise disjoint and $f(x_1) = \dots = f(x_r)$ as well as $g_k(x_1) = \dots = g_k(x_r)$ for 
$1\le k\le (r-1)d$; that is, the lemma does not guarantee equality for $g_0$. However, as 
$g_0+\dots +g_{(r-1)d}=1$ 
we also obtain $g_0(x_1) = \dots = g_0(x_r)$.
  
Suppose that for some $k$, the face $\sigma_j$ has at least one vertex in $C_k$. 
As we may again assume that $\sigma_j$ is the minimal face of $\Delta_N$ that contains $x_j$, 
this implies that $g_k(x_j) \neq 0$ and hence $g_k(x_i) \neq 0$ for $1\le i\le r$. 
Thus all $r$ faces $\sigma_i$ have at least one vertex in $C_k$. However, as $|C_k|=r$ and the $\sigma_i$ are pairwise disjoint, 
every $\sigma_i$ has exactly one vertex in $C_k$. Since this is true for every color, the $\sigma_i$ belong to the rainbow complex.
  
Thus the numbers $g_k(x_i)$ for $0\le k\le (r-1)d$ are exactly the barycentric coordinates of $x_i$. These are equal for all the $x_i$ since $g_k(x_1) = \dots = g_k(x_r)$ for all~$k$.
\end{proof}

Sober\'on in \cite[Section 4]{soberon:equal_coefficients} suggests an alternative idea for how to obtain the 
topological analogue of his result that we have obtained using our ansatz. 

The special case $r=2$ of Theorem \ref{theorem:equal_coeff},
which also establishes the B\'ar\'any--Larman conjecture \cite{barany1992colored} for $r=2$, 
is equivalent to the Borsuk--Ulam Theorem:

\begin{corollary} [(Borsuk--Ulam)]
For any continuous map $f : \partial\,\Diamond_{d+1} \to \IR^d$ from the boundary of the $(d+1)$-dimensional crosspolytope 
$\Diamond_{d+1}=\{x\in\IR^{d+1}:\sum_{i=1}^{d+1}|x_i|\le1\}$
to $\IR^d$, there are two points $x_1, x_2\in\Diamond_{d+1}$  
with $f(x_1) = f(x_2)$ that lie in opposite faces of~$\Diamond_{d+1}$ with equal barycentric coordinates, that is, $x_1 = -x_2$.
\end{corollary}

\begin{proof}
 For $r = 2$, Theorem~\ref{theorem:equal_coeff} treats a simplex with $2d+2$ vertices in $d+1$ color classes of cardinality~$2$. 
 Thus here the rainbow complex is exactly the boundary of the $(d+1)$-dimensional crosspolytope.
\end{proof}

We also obtain Sober\'on's original result in the same way.

\begin{theorem}[(Tverberg with equal barycentric coordinates: Sober\'on \cite{soberon:equal_coefficients})] \label{theorem:equal_coeff_affine}
Let $r \ge 2$, $d \ge 1$, and $N = N_{(r-1)d} = r((r-1)d+1)-1$. Let $f : \Delta_N \to \IR^d$ be affine. 
If the vertices of $\Delta_N$ are partitioned into $(r-1)d+1$ color classes of size $r$, 
then there are points $x_1, \dots, x_r$ with equal barycentric coordinates in $r$ pairwise disjoint rainbow faces 
$\sigma_1, \dots, \sigma_r$ of~$\Delta_N$ whose images intersect, with $f(x_1) = \dots = f(x_r)$.
\end{theorem}

\begin{proof}
The proof is the same as for Theorem \ref{theorem:equal_coeff}. 
Here the constraint functions are affine, so the proof follows by Remark~\ref{remark:affine}.
\end{proof}

Sober\'on shows that the number of color classes and the number of points per color class are optimal for the theorem to hold.
Thus the topological version, Theorem~\ref{theorem:equal_coeff}, is also optimal in that sense.

\section{Optimal colored versions of the topological Tverberg theorem}\label{sec:optimal}

The following theorem is a strengthening of the topological Tverberg theorem \ref{thm:top_tverberg_thm}
in the case when $r$ is a prime.
 
\begin{theorem}[(Optimal colored Tverberg: Blagojevi\'c, Matschke \& Ziegler \cite{blagojevic2009optimal})] \label{theorem:optimal}
Let $r\ge2$ be a prime, $d\ge1$, and $N\geq N_0=(r-1)(d+1)$. 
Let the vertices of $\Delta_N$ be colored by $m+1$ colors $C_{0}, \dots, C_m$ with $|C_i|\le r-1$ for all $i$. 
Then for every continuous map $f : \Delta_N \to \IR^d$ there are $r$ pairwise disjoint rainbow faces $\sigma_1, \dots, \sigma_r$ of~$\Delta_N$, such that $f(\sigma_1) \cap \dots \cap f(\sigma_r) \neq \emptyset$. 
\end{theorem}

As mentioned in the introduction, this theorem can as well be used as a black box result, to which the method
provided in Sections~\ref{sec:constraining_function} and \ref{sec:T-unavoidable-subcomplexes}
can be applied. In this section we will give two examples.

\begin{theorem}\label{theorem:optimal2}
Let $r\ge2$ be a prime, $d\ge1$, $\ell\ge0$, and $k\ge0$.
Let the vertices of $\Delta_N$ be colored by $\ell+k$ colors $C_{0}, \dots, C_{\ell+k-1}$ with 
$|C_0|\leq r-1$, \ldots, $|C_{\ell-1}|\leq r-1$ and \mbox{$|C_{\ell}|\geq 2r-1$}, \ldots, $|C_{\ell+k-1}|\geq 2r-1$,
where 
$|C_0|+\dots+ |C_{\ell-1}| > (r-1)(d-k+1)-k$. 
Then for every continuous map $f:\Delta_N\rightarrow\IR^d$ there are $r$ pairwise disjoint rainbow faces $\sigma_1, \dots, \sigma_r$ of~$\Delta_N$ such that $f(\sigma_1) \cap \dots \cap f(\sigma_r) \neq \emptyset$.
\end{theorem}

\begin{proof}
Without loss of generality we can assume that $|C_{\ell}|=\dots=|C_{\ell+k-1}|=2r-1$, by deleting any additional vertices.
Then the simplex $\Delta_N$ has still $N+1=|C_0|+\dots+ |C_{\ell-1}| +k(2r-1)$ vertices, so
\[
N = |C_0|+\dots+ |C_{\ell-1}| +k(2r-1)-1 \ge (r-1)(d+k+1)=N_k.
\]
Now we split each of the color classes 
$C_{\ell},\dots,C_{\ell+k-1}$ into new color sub-classes of cardinality at most $r-1$.
(For example, singletons will do.) 
Let $\Sigma_i$ be the subcomplex of all faces of $\Delta_N$
with at most one vertex in $C_i$.
Thus Theorem~\ref{theorem:optimal} together with the proof technique of Theorem~\ref{theorem:multiple_unavoidable}
yields that there is a Tverberg $r$-partition $\sigma_1,\dots,\sigma_r$
where each of the simplices $\sigma_i$ is a rainbow simplex with respect to the refined coloring
where the large color classes have been split into sub-classes, and it also 
lies in $\Sigma_{\ell}\cap\dots\cap \Sigma_{\ell+k-1}$, that is, it
uses at most one of the color sub-classes of each of $C_{\ell},\dots,C_{\ell+k-1}$ and thus
respects the original coloring. 
\end{proof}

This Theorem~\ref{theorem:optimal2} contains Theorem~\ref{theorem:optimal} as the special case $k=0$, 
and also 
Vre\'cica \& \v{Z}ivaljevi\'c's \cite[Prop.~5]{vrecica_zivaljevic:chessboard_complexes}
as the special case $|C_0|=\dots=|C_{\ell-1}|=r-1$ and $|C_{\ell}|=\dots=|C_{\ell+k-1}|=2r-1$.
If we further specialize to $r=2$ and $k=1$, this in turn reduces to the ``colored Radon''
Corollary~\ref{cor:colored_Radon}, as noted in \cite[Cor.~7]{vrecica_zivaljevic:chessboard_complexes}.
For $\ell=0$ we get Vre\'cica \& \v{Z}ivaljevi\'c's colored Tverberg theorem ``of type B,'' 
see \cite{vrecica1994new} and \cite[Cor.~8]{vrecica_zivaljevic:chessboard_complexes}.
\medskip

As a second instance of combining Theorem~\ref{theorem:optimal} with the proof technique of Theorem~\ref{theorem:multiple_unavoidable},
we finally obtain from our method the following new result about colored Tverberg partitions with restricted dimensions.
 
\begin{theorem}
Let $r\ge2$ be a prime, $d \ge 1$, $N\ge N_1 = (r-1)(d+2)$, and $k \ge \lceil\frac{r-1}{r}d\rceil$. 
Let the vertices of $\Delta_N$ be colored by $m+1$ colors $C_{0}, \dots, C_m$ with $|C_i|\le r-1$ for all $i$. 
Then for every continuous map $f : \Delta_N \to \IR^d$ there are $r$ pairwise disjoint rainbow faces $\sigma_1, \dots, \sigma_r$ 
of~$\Delta_N$ with $\dim\sigma_i \le k$ for $1\le i\le r$, such that $f(\sigma_1) \cap \dots \cap f(\sigma_r) \neq \emptyset$.
\end{theorem}

\providecommand{\noopsort}[1]{}
\providecommand{\bysame}{\leavevmode\hbox to3em{\hrulefill}\thinspace}

\begin{acknowledgements}\label{ackref}
    We are grateful to Moritz Firsching and Albert Haase for very valuable discussions and observations.
    Thanks to a referee for excellent comments and recommendations.
\end{acknowledgements}

\affiliationone{Pavle V. M. Blagojevi\'c\\
   Inst.\ Mathematics, FU Berlin\\ Arnimallee 2, 14195 Berlin\\ Germany 
   \\[1mm]
   and\\[1mm]
   Mathemati\v cki Institut SANU\\ Knez Mihailova 36, 11001 Beograd\\ Serbia\\
   \email{blagojevic@math.fu-berlin.de\\pavleb@mi.sanu.ac.rs}}
\affiliationtwo{Florian Frick\\Inst. Mathematics, MA 8-1, TU Berlin\\ Str. des 17. Juni 136, 10623 Berlin\\ Germany\\
    \email{frick@math.tu-berlin.de}}
\affiliationthree{G\"unter M. Ziegler\\
       Inst.\ Mathematics, FU Berlin\\ Arnimallee 2, 14195 Berlin\\ Germany\\  
       \email{ziegler@math.fu-berlin.de}}
\end{document}